\documentclass[12pt,a4paper]{article}
\usepackage[unicode=true,bookmarks=false]{hyperref}
\usepackage{amsmath,amsthm,amssymb,mathbbol}

\newcommand\address[1]{\def\theaddress{#1}}
\newcommand\addressii[1]{\def\theaddressii{#1}}
\newcommand\email[1]{\href{mailto:#1}{#1}}

\title{Cohomology for almost Lie algebroids}
\author{Melchior Gr\"utzmann\footnotemark[1]\, and Xiaomeng Xu\footnotemark[2]\\
  \normalsize\email{melchiorG@gMail.com},  \email{xuxiaomeng@pku.edu.cn}}
\address{Department of Mathematics, Northwestern Polytechnical University,  Xi'an 710129, China}
\addressii{Department of Mathematics and LMAM, Peking University,  Beijing 100871, China}

\newtheorem{defn}{Definition}[section]
\newtheorem{prop}[defn]{Proposition}
\newtheorem{cor}[defn]{Corollary}
\newtheorem{thm}[defn]{Theorem}
\newtheorem{lemma}[defn]{Lemma}
\theoremstyle{definition}
\newtheorem{ex}[defn]{Example}
\newtheorem{rem}[defn]{Remark}
\newcommand\defbb[2]{\def#1{{\mathbb{#2}}}}
\newcommand\defcal[2]{\def#1{{\mathcal{#2}}}}
\newcommand\deffrak[2]{\def#1{{\mathfrak{#2}}}}
\newcommand\defrm[2]{\def#1{{\mathrm{#2}}}}

\newcommand\XX[1]{\textbf{XX: #1}}

\def\<{\langle}
\def\>{\rangle}
\def\[{\begin{equation}}
\def\]{\end{equation}}
\def\:{\colon}
\defcal\A{A}
\def\conn_#1{\nabla_{\!#1\,}}  
\newcommand\cp{{\text{c.p.}}} 
\defrm\ud{d}
\DeclareMathOperator\Der{Der}
\deffrak\g{g}
\DeclareMathOperator\Graph{graph}
\newcommand\haszero[1][enumi]{\addtocounter{#1}{-1}}
\DeclareMathOperator\id{id}
\DeclareMathOperator\im{im}  
\defbb\k{k}  
\defcal\Lie{L}  
\defcal\M{M}
\defbb\N{N}
\newcommand\pfrac[2]{\frac{\partial #1}{\partial #2}}
\DeclareMathOperator\pt{pt}
\defbb\R{R}
\defcal\cR{R}  
\DeclareMathOperator\rk{rk}
\newcommand\subref[2]{\hyperref[#2]{\ref*{#1}-\ref*{#2}}}
\DeclareMathOperator\Symm{S}  
\DeclareMathOperator\sgn{sgn}  
\newcommand\smooth{\mathrm{C}^\infty}
\defcal\gsmooth{O}
\providecommand\strong[1]{\textbf{#1}}
\deffrak\X{X}  
\newcounter{xctr}
\defbb\Z{Z}

\begin{document}
\maketitle
{\renewcommand{\thefootnote}{\fnsymbol{footnote}}
\footnotetext[1]{\theaddress}
\footnotetext[2]{\theaddressii}
}
\abstract{We give numerous examples of almost Lie algebroids arising as Dirac structures in pre-Courant algebroids, e.g.\ from twisted Poisson structures, as well as from twisted actions of a Lie algebra.  We moreover define a cohomology for them, motivated by a Q-structure, that is trivial for (bundles of) Lie algebras but characterizes the underlying integrable distribution the image of the anchor map.
}

\section{Introduction}
Almost Lie algebroids are a generalization of Lie algebroids.  While the latter were introduced in \cite{Pra67} and have been broadly studied, see e.g.\ Mackenzie \cite{Mack05} and references therein, and permit an integration theory via Lie groupoids (see Mackenzie or \cite{CrF01} for integrability conditions) and cohomology theory (see Pradines), the former have just recently come into the interest of physicists and mathematicians in the framework of nonholonomic mechanics, see e.g.\ \cite{Cor06}.  In this paper we wish to further study their structures, namely we will introduce a cohomology theory for them and give numerous examples of almost Lie algebroids arising as Dirac structures in pre-Courant algebroids, e.g.\ from twisted Poisson structures, as well as from twisted actions of a Lie algebra.  Previous work has been done in \cite{Gru10}, but analogously to an observation in \cite{LSX12a} the weak integrability condition on the Jacobiator is unnecessary, i.e.\ automatically fulfilled.  The details can be found in Section~\ref{s:Jacob}.

The definition of the cohomology is guided by supergeometry, namely the same Q-structure as in Section~5 of \cite{Gru10}, however here we take the smaller realization on superfunctions and give explicit formulas in terms of smooth geometry, \XX{What is it?}Theorem~\ref{t:cohom}.  The reader is thus not quired to understand supergeometry, except for Subsection~\ref{ss:super} where we make the relation explicit.  
The result may thus also be interesting for topological field theory, because the underlying Q-structure permits via \cite{AKSZ97} the construction of a topological field theory out of an almost Lie algebroid.

It turns out that the cohomology is trivial for almost Lie algebras.  However for an (almost) Lie algebroid with non-vanishing anchor map it is in general nontrivial.  In particular for the tangent bundle of a smooth manifold it is isomorphic to the de Rham cohomology.  More generally for regular almost Lie algebroids (i.e.\ the anchor map has constant rank and its image is thus an integrable distribution), it computes the Lie algebroid cohomology of the underlying distribution, see Proposition~\ref{p:cohom}.  It suggests thus a definition of cohomology for singular but integrable distributions, at least in the case where these can be seen as the image of the anchor map of an almost Lie algebroid.\par\bigskip

The structure of the paper is as follows.  In section~\ref{s:prelim} we briefly summarize the needed geometric structure of an almost Lie algebroid together with its Jacobiator and exterior covariant derivative. We also give numerous examples of almost Lie algebroids arising as Dirac structures in pre-Courant algebroids, e.g.\ from twisted Poisson structures, as well as from twisted actions of a Lie algebra.  In Section~\ref{s:Jacob} we show that the Jacobiator is indeed closed under the exterior covariant derivative.  This permits us to define in Section~\ref{s:cohom} a cohomology for almost Lie algebroids in terms of geometric structures.  Examples and some properties of the cohomology are shown in Section~\ref{s:ex}.  In the appendix \ref{a:Nmfd} we give a minimalistic introduction to $\N$-manifolds and Q-structures.

\section{Almost Lie algebroids}\label{s:prelim}
\begin{defn}\label{d:aLie}  An almost Lie algebroid is a vector bundle $A\to M$ over a smooth manifold $M$ together with a skew-symmetric operation $[\cdot,\cdot]\:\Gamma(A)\wedge\Gamma(A)\to\Gamma(A)$ called the bracket and a morphism of vector bundles $\rho\:A\to TM$ called the anchor subject to the rules
\begin{align}
 [\phi,f\cdot\psi] &= \rho(\phi)[f]\cdot\psi +f\cdot[\phi,\psi], \label{Leibniz} \\
 \rho[\phi,\psi] &= [\rho(\phi),\rho(\psi)] \label{morph}
\end{align} for every $\phi,\psi\in\Gamma(A)$ and $f\in\smooth(M)$.

We call $(A,\rho,[\cdot,\cdot])$ regular iff $\im\rho\subset TM$ is a vector bundle.
\end{defn}
\eqref{Leibniz} is called Leibniz rule of the algebroid and \eqref{morph} means that $\rho$ is a morphism of brackets from $(\Gamma(A),[\cdot,\cdot])$ to $(\Gamma(TM),[\cdot,\cdot])$.

\begin{ex}\label{x:1}\begin{enumerate}\haszero\item Lie algebroids $(A,\rho,[\cdot,\cdot])$ are a kind of trivial examples of almost Lie algebroids.  They fulfill in addition the Jacobi identity
 \[ [\phi,[\psi,\chi]]+[\psi,[\chi,\phi]]+[\chi,[\phi,\psi]] = 0
 \] for all $\phi,\psi,\chi\in\Gamma(A)$.  
  \begin{enumerate}
  \item A particular case of Lie algebroid is the tangent bundle $A=TM$ with $\rho=\id$ and $[\cdot,\cdot]$ the commutator bracket of vector fields.  This is a regular Lie algebroid.
  \item To see that there are also regular Lie algebroids with non-trivial $\ker\rho$, pick a Lie algebra $(\g,[\cdot,\cdot]_\g)$ and on the bundle $A:=TM\times\g$ define the bracket
 $$ [X\oplus f\xi,Y\oplus g\eta]:= [X,Y]\oplus X[g]\xi-Y[f]\eta+fg[\xi,\eta]_\g
 $$ for $X,Y\in\Gamma(TM)$, $f,g\in\smooth(M)$, and $\xi,\eta\in\g$.  It is easy to see that this is a regular Lie algebroid with $\rho(X\oplus f\xi)=X$ and thus $\ker\rho=M\times\g$.
  \end{enumerate}
\item Starting from a regular Lie algebroid $(A,\rho,[\cdot,\cdot]_0)$, setting $F:=\ker\rho$ a vector bundle and choosing $B\in\Omega_M^2(A,F):=\Gamma(\wedge^2 A^*\otimes F)$, we see that
 \[ [\phi,\psi]_B:= [\phi,\psi]_0+B(\phi,\psi)
 \] together with $A$ and $\rho$ fulfills the axioms of an almost Lie algebroid.
\item\label{x:aLiealg}\setcounter{xctr}{\value{enumi}}  An almost Lie algebra is an almost Lie algebroid over a point $\pt$.  In particular the anchor map vanishes and the conditions \eqref{Leibniz} and \eqref{morph} are thus trivial.  We are thus dealing with an arbitrary skew-symmetric operation $[.,.]\:\wedge^\g\to\g$, where we denote the vector space $A=\g$.
\end{enumerate}
\end{ex}

\begin{rem}  According to \cite{LSX12a}, a pre-Courant algebroid $(E,(\cdot,\cdot),\rho,[\cdot,\cdot])$ is a vector bundle $E\to M$ with a non-degenerate symmetric bilinear form $(\cdot,\cdot)$, a bundle morphism $\rho\:E\to TM$, and a bracket $[\cdot,\cdot]\:\Gamma(E)\otimes\Gamma(E)\to\Gamma(E)$ with properties such as
\begin{align*}
  [\phi,f\cdot\psi] &= \rho(\phi)[f]\cdot\psi +f\cdot[\phi,\psi],  \\
  \rho[\phi,\psi] &= [\rho(\phi),\rho(\psi)], \\
  [\phi,\phi] &= \tfrac12\rho^*\ud\<\phi,\phi\>
\end{align*} for all $\phi,\psi\in\Gamma(E)$ and $\rho^*\:T^*M\to E\cong E^*$ its transpose where the latter identification is via the non-degenerate inner product $(\cdot,\cdot)$.  Note that $[\cdot,\cdot]$ is a Dorfman type bracket, i.e.\ in particular not skew-symmetric.

%
Given such a pre-Courant algebroid $(E,(\cdot,\cdot),\rho,[\cdot,\cdot])$
, then we can define 
sub-Dirac structures as isotropic subbundles $L\subset E$, i.e.\ $(L,L)=0$ for which the bracket closes as $[\Gamma(L),\Gamma(L)]\subset\Gamma(L)$.  Then $L$ is an almost Lie-algebroid, because the violation of skew-symmetry vanishes.
\end{rem}

\begin{ex}\label{x:2}\begin{enumerate}\addtocounter{enumi}{\value{xctr}}\item One particular source for pre-Courant algebroids are Cartan geometries.  Certain isotropic Lie subalgebras of the Lie algebra of the bigger structure group induce sub-Dirac structures in these twisted Courant algebroids
.  Details can be found in \cite{XX12b}.

\item\label{x:tPoisson}  Let $\Pi\in\Gamma(\wedge^2 TM)$ be a twisted Poisson bivector, i.e.\ there is a 3-form $H\in\Omega^3(M)$ (not necessarily closed) such that $\tfrac12[\Pi,\Pi]=\Pi^\#(H)$.  According to \cite{SL10b} we can define an almost Lie bracket on $T^*M\cong\Graph\Pi^\#$ by considering the latter as a Dirac structure in the twisted Courant algebroid $(TM\oplus T^*M)_H$.  This means in particular
\[ [\alpha,\beta]_{\Pi,H} = \Lie_{\Pi^\#\alpha}\beta -\Lie_{\Pi^\#\beta}\alpha +\ud(\Pi(\alpha,\beta)) +H(\Pi^\#\alpha,\Pi^\#\beta,\cdot)
\]  For $\alpha,\beta\in\Gamma(T^*M)$.  The anchor map is $\Pi^\#\:T^*M\to TM$ and a morphism of brackets.  Also the Leibniz rule is obviously fulfilled.  Therefore twisted Poisson structures give almost Lie algebroids.
\nocite{Sev01}

\item  Following \cite{LSX12a} a twisted action of a Lie algebra $\g$ on a smooth manifold $M$ is given by a pair of bundle maps $(\rho,k)$
 $$\rho\:M\times\g\to TM; \quad k\:M\times\wedge^2\g\to M\times\g
 $$ satisfying
\begin{align}
 k(e,\cdot) &=0,  \quad \forall e\in\ker(\rho), \\
 \label{b} \rho([e_1,e_2]_{M\times\g}) &= [\rho(e_1),\rho(e_2)]-\rho(k(e_1,e_2)),
\intertext{for all $e_1,e_2\in \Gamma(M\times\g)$ and the induced bracket on $\Gamma(M\times \g)$}
  [e_1,e_2]_{M\times\g} &=\Lie_{\rho(e_1)}e_2-\Lie_{\rho(e_2)}e_1+[e_1,e_2]_\g,  \nonumber
\end{align}
 where $[e_1,e_2]_\g$ is the pointwise Lie bracket of two sections $e_1,e_2\in\Gamma(M\times\g)=\smooth(M,\g)$ and $\Lie$ denotes the action of vector fields on the trivial vector bundle $M\times\g$.

Just as Lie algebroids can arise from Lie algebra actions, we have almost Lie algebroids arising from these twisted actions in the following way:

Twist the bracket $[\cdot,\cdot]_{M\times\g}$ by $k$ as follows
 \[
 [e_1,e_2]_{k}=[e_1,e_2]_{M\times\g}+k(e_1,e_2).
 \]
 Equation \eqref{b} implies that $\rho$ is a morphism of brackets from this twisted bracket $[\cdot,\cdot]_{k}$ and also the Leibniz rule still holds.  Therefore $(M\times\g,[\cdot,\cdot]_k,\rho)$ is an almost Lie algebroid.  A nontrivial example of a twisted action can arise in Cartan geometry, in this case the underlying vector bundle is just the tangent bundle, see \cite{LX12} for more details about that.
\end{enumerate}
\end{ex}

Arbitrary almost Lie algebroids are in general no Lie algebroids.  To measure the difference, we introduce the following notion:
\begin{defn}  Given an almost Lie algebroid $(A,\rho,[\cdot,\cdot])$, then its Jacobiator is the map $J\:\wedge^3\Gamma(A)\to\Gamma(A)$,
 \[ J(\phi,\psi,\chi):= [\phi,[\psi,\chi]] +[\psi,[\chi,\phi]] +[\chi,[\phi,\psi]].
 \]
\end{defn}
It is easy to see that $J$ is $\smooth$-linear in its three arguments and moreover $\rho\circ J=0$.  Therefore $J\in\Omega_M^3(A,\ker\rho)$.

Remember that $(A,\rho,[\cdot,\cdot])$ is a Lie algebroid iff $J=0$.\par\bigskip

It is also possible to generalize the computations to non-regular almost Lie algebroids.  Analogous to \cite[Chap.~2, Rem.~3]{Gru10} we give the following definition:  Let $\k$ be a field and $\cR$ be an associative commutative unital algebra over $\k$.  Tensor products no further specified are taken over $\cR$, while $\otimes_\k$ is the tensor product over $\k$.  Let further $\Der(R)$ denote the $\k$-linear derivations of $\cR$.
\begin{defn}\label{d:aLR}  An almost Lie--Rinehart algebra $(\cR,\A,\rho,[\cdot,\cdot])$ is a ring $\cR$ together with a module $\A$, a module homomorphism $\rho\:\A\to\Der(\cR)$, and a $\k$-bilinear skew-symmetric bracket $[\cdot,\cdot]\:\A\otimes_\k\A\to\A$ subject to
\begin{align*}
  \rho[\phi,\psi] &= [\rho(\phi),\rho(\psi)], \\
  [\phi, f\cdot\psi] &= \rho(\phi)[f]\cdot\psi +f\cdot[\phi,\psi]
\end{align*} for all $\phi,\psi\in\A$ and $f\in\cR$.
\end{defn}
Obviously the sections $\Gamma(A)$ of an almost Lie algebroid $(A,\rho,[\cdot,\cdot])$ form an almost Lie--Rinehart algebra.  It is a Lie--Rinehart algebra as introduced by Rinehart in \cite{Rin63} iff in addition $J=0$ for the analogous definition of the Jacobiator.

\begin{rem}\label{r:sing}
We now define the ``smooth sections $\Gamma(\ker\rho)$'' as $v\in\A$ with $\rho(v)=0$.  Correspondingly we define ``$\Omega_M^p(A,\Symm^q\ker\rho)$'' as the module of all $v\in\Omega_M^p(A,\Symm^qA):=\wedge^q\A^*\otimes \Symm^q\A$ such  that $\tilde\rho(v)=0$ for $\tilde\rho\:\wedge^\bullet\A^*\otimes \Symm^q\A\to \wedge^\bullet\A^*\otimes \Symm^{q-1}\A\otimes\Der(\cR): \alpha\otimes\psi_1\psi_2\mapsto \alpha\otimes(\psi_2\otimes\rho(\psi_1)+\psi_1\otimes\rho(\psi_2))$ for $\alpha\in\wedge^\bullet\A^*$, $\psi_i\in \A$, extended linearly and correspondingly for more factors $\psi_i$.  By $\Gamma((\ker\rho)^*)$ we mean the elements of the dual module to ``$\Gamma(\ker\rho)$''.

In what follows we will stick to the notion $\Gamma(\ker\rho)$ and $\Omega_M^p(A,\ker\rho)$ correspondingly to keep similarity with the usual differential geometry of vector bundles.  The attentive reader will however easily be able to replace this with the corresponding notion for almost Lie--Rinehart algebras.
\end{rem}

Another important notion is the following generalization of a connection.
\begin{defn}  Given a vector bundle $A\to M$ together with a morphism $\rho\:A\to TM$ and a second vector bundle $V\to M$ over the same base.  An $A$-connection $\nabla$ on $V$ is an $\R$-linear map $\nabla\:\Gamma(A)\otimes\Gamma(V)\to\Gamma(V)$ subject to the rules
\begin{align}
  \conn_{f\phi} v &= f\conn_\phi v, \\
  \conn_\phi (f\cdot v) &= \rho(\phi)[f]\cdot v +f\cdot\conn_\phi v
\end{align} for every $\phi\in\Gamma(A)$, $v\in\Gamma(V)$ and $f\in\smooth(M)$.
\end{defn}
\begin{ex}\begin{enumerate}\item  Given a vector bundle $V\to M$, then a usual connection $\nabla$ on $V$ is a $TM$-connection on $V$.  In particular the Levi-Civita connection of Riemmanian geometry is a $TM$-connection on $TM$.
\item  Given a morphism $\rho\:A\to TM$ and a trivial vector bundle $V$, then it is possible to construct a connection as follows.  Let $e_b\in\Gamma(V)$, $b=1,\dots,\rk V$ be a global frame of $V$ and define
 \[ \nabla_\phi(v^be_b):= \rho(\phi)[v^b]\cdot e_b
 \] where $v^b\in\smooth(M)$ and $\phi\in\Gamma(A)$ and we use Einstein's sum convention.

Starting with an arbitrary $V$ and given a smooth partition of unity on $M$ subordinate to a cover of charts of $M$ trivializing $V$, then it is possible to glue the above construction over the charts to a connection on $V$.
\end{enumerate}
\end{ex}

\begin{prop}  Given an almost Lie algebroid $(A,\rho,[\cdot,\cdot])$, then the operator $D$ for $\alpha\in\Omega_M^p(A):=\Gamma(\wedge^p A^*)$ and $\psi_i\in\Gamma(A)$
\[\begin{split} D\alpha(\psi_0,\dots,\psi_p) :=&\, \sum_{i=0}^p (-1)^i \rho(\psi_i)[\alpha(\psi_0,\dots,\hat\psi_i,\dots,\psi_p)] \\
  &+\sum_{i<j} (-1)^{i+j}\alpha([\psi_i,\psi_j],\psi_0,\dots,\hat\psi_i,\dots,\hat\psi_j,\dots,\psi_p)
\end{split}  \label{D}
\] maps $\Omega_M^p(A)\to\Omega_M^{p+1}(A)$.
\end{prop}
This can be proved with straight-forward computations analogous to the Lie algebroid differential on $TM$.

Starting from here we will fix a (regular) almost Lie algebroid $(A,\rho,[\cdot,\cdot])$ and denote $F:=\ker\rho\subset A$ as well as $t\:F\hookrightarrow A$ the corresponding embedding.

The following was already observed in \cite{Gru10}.
\begin{prop}  Given the situation above, then
 \[ t(\conn_\phi v):=[\phi,t(v)]
 \] for $\phi\in\Gamma(A)$ and $v\in\Gamma(F)$ is an $A$-connection on $F$.
\end{prop}

\begin{cor}  Given the notation of the previous proposition, then $D$ can be extended to $\alpha\in\Omega_M^p(A,F)$ via
\[\begin{split} D\alpha(\psi_0,\dots,\psi_p) :=&\, \sum_{i=0}^p (-1)^i \conn_{\psi_i}(\alpha(\psi_0,\dots,\hat\psi_i,\dots,\psi_p)) \\
  &+\sum_{i<j}(-1)^{i+j} \alpha([\psi_i,\psi_j],\psi_0,\dots,\hat\psi_i,\dots,\hat\psi_j,\dots,\psi_p)
\end{split}
\]
\end{cor}

Note that $D$ in general does not square to 0.  Namely $D$ on $\Omega_M^\bullet(A)$ squares to 0 iff the Jacobiator vanishes.  In this case $D$ is just the Lie algebroid differential.  Since $F\subset A$ the vanishing of the Jacobiator also implies that $D^2$ vanishes on $\Omega_M^\bullet(A,F)$.

The connection $\nabla$ naturally extends to $F^*$ via
\[ \<\conn_\psi\beta,v\>_F = \rho(\psi)\<\beta,v\>_F-\<\beta,\conn_\psi v\>_F.
\] for $\psi\in\Gamma(A)$, $\beta\in\Gamma(F^*)$, and $v\in\Gamma(F)$.  $\<\cdot,\cdot\>_F$ denotes the canonical pairing between $F$ and $F^*$.

Analogously also $D$ extends to $\Omega_M^\bullet(A,F^*)$.  We further extend $\nabla$ to $\Symm^\bullet F^*$ the symmetric algebra over $F$ via the Leibniz rule
\[ \conn_\psi (\alpha\cdot\beta)= \conn_\psi\alpha\cdot\beta +\alpha\cdot\conn_\psi\beta
\] for $\alpha,\beta\in\Gamma(F^*)$ and $\psi\in\Gamma(A)$ and correspondingly for more factors.  $D$ is extended in an analogous way.

Note that $D^2$ still does not square to 0 unless the Jacobiator vanishes.

From \cite{Gru10} we also take the notion of H-twisted Lie algebroid.
\begin{defn}\label{d:HLie}  An H-twisted Lie algebroid $(A,\rho,[\cdot,\cdot],J)$ is an almost Lie algebroid such that the Jacobiator fulfills $DJ=0$.
\end{defn}
In the next section we will examine what is the content of this additional condition.

\section{Relation to H-twisted Lie algebroids}\label{s:Jacob}
Throughout this section $(A,\rho,[\cdot,\cdot])$ will be an almost Lie algebroid. and $J\:\wedge^3A\to\ker\rho$ its Jacobiator.  Remember the notion of H-twisted Lie algebroid in Definition~\ref{d:HLie} introduced in \cite{Gru10}.  We can now simplify its definition as follows.
\begin{thm}  Every almost Lie algebroid is an H-twisted Lie algebroid, i.e.\ $\rho\circ J=0$, $J\in\Omega_M^3(A,\ker\rho)$ and $DJ=0$.
\end{thm}
\begin{proof}  We already know that $J\in\Omega_M^3(A,\ker\rho)$.  We can therefore apply the exterior covariant derivative $D$ as follows.
Let $\psi_i\in\Gamma(A)$ be given and compute
\begin{align*}
  \<DJ,\psi_1\wedge\psi_2&\wedge\psi_3\wedge\psi_4\> = \conn_{\psi_1}J(\psi_2,\psi_3,\psi_4) -J([\psi_1,\psi_2],\psi_3,\psi_4) +\cp \\
  &= \sum_{\sigma\in A_4} [\psi_{\sigma1},[\psi_{\sigma2},[\psi_{\sigma3},\psi_{\sigma4}]]] -\sum_{\sigma\in A_4} [[\psi_{\sigma1},\psi_{\sigma2}],[\psi_{\sigma3},\psi_{\sigma4}]] -\dots\\
  &\quad -\sum_{\sigma\in A_4}[\psi_{\sigma3},[\psi_{\sigma4},[\psi_{\sigma1},\psi_{\sigma2}]]]
\intertext{where $S_4$ is the symmetric group in 1,2,3,4 and $A_4\subset S_4$ the even permutations.}
  &= \sum_{\sigma\in A_4} [[\psi_{\sigma1},\psi_{\sigma2}],[\psi_{\sigma3},\psi_{\sigma4}]]  \\
  &= \tfrac12\sum_{\sigma\in S_4} \sgn(\sigma) [[\psi_{\sigma1},\psi_{\sigma2}],[\psi_{\sigma3},\psi_{\sigma4}]]
\intertext{where $\sgn\:S_4\to\{\pm1\}$ is the sign of a permutation.}
  &= \tfrac14\sum_{\sigma\in S_4} \sgn(\sigma)\big([[\psi_{\sigma1},\psi_{\sigma2}],[\psi_{\sigma3},\psi_{\sigma4}]] - [[\psi_{\sigma3},\psi_{\sigma2}],[\psi_{\sigma1},\psi_{\sigma4}]] \big)  \\
  &= \tfrac14\sum_{\sigma\in S_4} \sgn(\sigma)\big([[\psi_{\sigma1},\psi_{\sigma2}],[\psi_{\sigma3},\psi_{\sigma4}]] + [[\psi_{\sigma3},\psi_{\sigma4}],[\psi_{\sigma1},\psi_{\sigma2}]] \big)
\intertext{where in the last two steps we exchanged two elements in the second summand.}
 &= \tfrac14\sum_{\sigma\in S_4} \sgn(\sigma)\big([[\psi_{\sigma1},\psi_{\sigma2}],[\psi_{\sigma3},\psi_{\sigma4}]] - [[\psi_{\sigma1},\psi_{\sigma2}],[\psi_{\sigma3},\psi_{\sigma4}]] \big)  \\
\intertext{where we exchanged the arguments of the outer bracket in the second summand.  But now it vanishes and the theorem is proven}
 &=0.
 \end{align*}
\end{proof}

\section{Another cohomology}\label{s:cohom}
The Theorem in the last section permits us to give another geometric construction of cohomology for almost Lie algebroids as stated in Theorem~\ref{t:cohom}.

Throughout this section let $(A,\rho,[\cdot,\cdot])$ be an almost Lie algebroid and $F:=\ker\rho$ be a (possibly singular) vector subbundle of $A$ -- $\Gamma(\ker\rho)$ defined as in Remark~\ref{r:sing}.  We denote again by $t\:F\hookrightarrow A$ the embedding.

We will show that the operator $D\:C^n\to C^{n+1}$ with
\[ C^n:=\bigoplus_{p+2q=n}\Gamma(\wedge^pA^*\otimes \Symm^q F^*)  \label{cochains}
\] can be modified such that it squares to 0.

First note that $D$ fits the following definition
\begin{defn}\label{d:odd}  An $\R$-linear map $D\:C^\bullet\to C^{\bullet+1}$ is said to be an odd derivative iff
\[\label{oddLeibniz}  D(\gamma_1\gamma_2) = D(\gamma_1)\gamma_2 +(-1)^{|\gamma_1|}\gamma_1D(\gamma_2).
\] for all $\gamma_1\in C^{|\gamma_1|}$ and $\gamma_2\in C^\bullet$.
\end{defn}
Here and in what follows $|\gamma|\in\N$ will denote the degree of a homogeneous $\gamma\in C^\bullet$.

\begin{rem}\label{r:gen}
Note that moreover $C^\bullet$ is an associative graded commutative algebra generated by $C^0=\smooth(M)$, $C^1=\Gamma(E^*)$ and $\Gamma(F^*)\subset C^2$.  This will simplify our computations in what follows.
\end{rem}

\begin{defn}  We define the map $\hat\delta\:C^n\to C^{n+1}$ via
 \[ \hat\delta(\alpha_1\wedge\alpha_2\otimes\beta) = \alpha_2\otimes t^*(\alpha_1)\beta -\alpha_1\otimes t^*(\alpha_2)\beta  \label{deltahat}
 \] for $\alpha_i\in A^*$ and $\beta\in F^*$ all over the same base point $x\in M$ and $t^*\:E^*\to F^*$ the transpose of $t$.  $\hat\delta$ is extended correspondingly to more or less factors in $E^*$.
\end{defn}
Note that also $\hat\delta$ is an odd derivative in the sense of Definition~\ref{d:odd}.  Moreover it squares to $\hat\delta^2=0$.

Remember that $J\in\Omega_M^3(A,F)$.  We can therefore define the map:
\begin{defn}  $\hat{J}\:C^n\to C^{n+1}$ via
\[  \hat{J}(\alpha\otimes\beta)= \alpha\wedge\<\beta,J(\cdot,\cdot,\cdot)\>_F   \label{Jhat}
\] for $\alpha\in\Omega_M^p(A)$ and $\beta\in\Gamma(F^*)$ and extended symmetrically for higher powers of $F^*$.
\end{defn}
Note that also $\hat{J}$ is an odd derivative and squares to 0.

The following observation from super geometry will be useful in calculations:
\begin{lemma}\label{l:commut}  Let $A,B\:C^\bullet\to C^{\bullet+1}$ two odd derivations, then $D:=[A,B]:=A\circ B+B\circ A$ maps $C^\bullet\to C^{\bullet+2}$ and fulfills the ungraded Leibniz rule
\[ D(\gamma_1\gamma_2) = D(\gamma_1)\gamma_2+\gamma_1D(\gamma_2)  \label{evenLeibniz}
\] for all $\gamma_i\in C^\bullet$.
\end{lemma}
\begin{proof}  straight-forward calculations.
\end{proof}

\begin{thm}\label{t:cohom}  Given a regular almost Lie algebroid $(A,\rho,[\cdot,\cdot])$ together with the cochains $C^n$ defined in \eqref{cochains}, the map $\hat{J}$ and $\hat\delta$ defined in \eqref{deltahat} and \eqref{Jhat}, then
\[ \ud:= D+\hat{J}+\hat\delta  \label{ud}
\] maps $\ud\:C^n\to C^{n+1}$, fulfills the Leibniz rule \eqref{oddLeibniz} and squares to 0.
\end{thm}
The proof will cover the remaining part of this section.

First note that the mapping behavior of $\ud\:C^n\to C^{n+1}$ and the Leibniz rule \eqref{oddLeibniz} follow from the corresponding behavior of the summands.  It remains therefore to check that
$$ \ud^2= D^2+(D\circ\hat\delta+\hat\delta\circ D)+(D\circ\hat{J}+\hat{J}\circ D)+(\hat\delta\circ\hat{J}+\hat{J}\circ\hat\delta)
$$ vanishes.  We will prove that in several steps.
\begin{lemma}  $D\circ\hat\delta +\hat\delta\circ D=0$
\end{lemma}
\begin{proof}
  By Lemma~\ref{l:commut} and Remark~\ref{r:gen} it is sufficient to check the claim for $f\in\smooth(M)$, $\alpha\in\Gamma(A^*)$ and $\beta\in\Gamma(F^*)$.  Note that $\hat\delta(f)=0$ according to its definition.  Let in addition $v\in\Gamma(F)$
\begin{align*}
  \<\hat\delta(Df),v\>_F &= \<Df,t(v)\> = \rho(t(v))[f]=0
\intertext{for $\alpha\in\Gamma(E^*)$ let further $\psi\in\Gamma(E)$}
 \<D(\hat\delta(\alpha)),\psi\otimes v\> &=\<\<D\circ\hat\delta(\alpha),\psi\>,v\>_F =\<\conn_\psi\hat\delta(\alpha),v\>_F \\
  &= \rho(\psi)\<\hat\delta(\alpha),v\>_F -\<\hat\delta(\alpha),\conn_\psi v\>_F  \\
  &= \rho(\psi)\<\alpha,t(v)\> -\rho(t(v))\<\alpha,\psi\> -\<\alpha,[\psi,t(v)]\>  \\
  &= \<D\alpha, \psi\wedge t(v)\> = -\<\hat\delta(D\alpha), \psi\otimes v\>
\end{align*}
For $\beta\in\Gamma(F^*)$ we have $\hat\delta(\beta)=0$ by definition.  Moreover $\hat\delta\circ D(\beta)(x)$ for $x\in M$ depends only on $\beta$ in a neighborhood of $x$.  But since $t^*\:E^*\to F^*$ is surjective we have an $\alpha\in\Gamma(A^*)$ over this neighborhood with $\beta=t^*\alpha=\hat\delta(\alpha)$.  Therefore $D\beta=D\circ\hat\delta(\alpha)=-\hat\delta\circ D\alpha$ by the previous computation and so $\hat\delta(D(\beta))=0$, because $\hat\delta^2=0$.  This completes the proof.
\end{proof}

\begin{lemma} $D\circ\hat{J} +\hat{J}\circ D=0$
\end{lemma}
\begin{proof}  Again by Lemma~\ref{l:commut} and Remark~\ref{r:gen} it is sufficient to check the claim for $f\in\smooth(M)$, $\alpha\in\Gamma(E^*)$ and $\beta\in\Gamma(F^*)$.  Note that $Df\in\Gamma(E^*)$ and therefore $\hat{J}(Df)=0$ as well as $\hat{J}(f)=0$.  Analogously $\hat{J}(\alpha)=0=\hat{J}(D\alpha)$.  It remains to observe
 $$ D(\hat{J}(\beta)) = D\<\beta,J\>_F = \<D\beta\wedge J\>_F +\<\beta, DJ\>_F = -\hat{J}(D\beta).
 $$
\end{proof}

Now we need an expression for $D^2$.  We introduce therefore the following two maps.
\begin{defn}  We define the map
 \begin{align} L\:\wedge^pA^*\otimes \Symm^qF &\to\wedge^{p+2}A\otimes \Symm^q F,\\
   J(\alpha\otimes v_1v_2) &= \alpha\wedge J(v_1,\cdot,\cdot) v_2 +\alpha\wedge J(v_2,\cdot,\cdot) v_1 \nonumber
 \end{align} for $\alpha\in\wedge^p A^*$ and $v_1,v_2\in F$ all over the same base point $x\in M$.  $L$ extends linearly to arbitrary elements of the tensor product and correspondingly to higher $q$.

 We define its partial transpose as
 \[ L^*\:C^n\to C^{n+2},\quad L^*(\alpha\otimes\beta)= \alpha\wedge L^*(\beta)
 \] for $\alpha\in\wedge^p A^*$, $\beta\in \Symm^\bullet F^*$ all over the same base point $x\in M$.
  For $\beta_1\in F^*$ and $v\in F$ also over $x$, $L^*$ is defined as
 \[\<L^*(\beta_1),v\>_F:=\<\beta_1,L(v)\>_F\in\wedge^2 A^* \]
 and extends to higher powers using an ungraded Leibniz rule.
\end{defn}
Note that $L^*$ also maps $C^n$ to $C^{n+2}$ and fulfills the ungraded Leibniz rule \eqref{evenLeibniz}.

\begin{defn} We define the map $\tilde{J}\:C^n\to C^{n+2}$ as
\[ \tilde{J}(\alpha_1\wedge\alpha_2\otimes\beta) = \alpha_1(t(J(\cdot,\cdot,\cdot)))\wedge\alpha_2\otimes\beta -\alpha_2(t(J(\cdot,\cdot,\cdot))\wedge\alpha_1\otimes\beta
\] for $\beta\in\Gamma(\Symm^\bullet F^*)$ and $\alpha_i\in\Gamma(E^*)$ and extended accordingly to higher powers in $\wedge^\bullet E^*$.
\end{defn}
Note that $\tilde{J}$ fulfills the ungraded Leibniz rule \eqref{evenLeibniz}.

\begin{prop}\label{p:D^2}
 \[ D^2= -\tilde{J} -L^*
 \]
\end{prop}
\begin{proof}  Again by Lemma~\ref{l:commut} and Remark~\ref{r:gen} it is sufficient to verify the formula for $f\in\smooth(M)$, $\alpha\in\Gamma(E^*)$ and $\beta\in\Gamma(F^*)$.  Let $\psi_i\in\Gamma(E)$ and $v\in\Gamma(F)$
\begin{align*}
  \<D^2f,\psi_1\wedge\psi_2\> =\,& \rho(\psi_1)\<Df,\psi_2\> -\rho(\psi_2)\<Df,\psi_1\> -\<Df,[\psi_1,\psi_2]\>  \\
  =& [\rho(\psi_1),\rho(\psi_2)][f] -\rho[\psi_1,\psi_2][f]
\intertext{and the latter expression vanishes due to axiom \eqref{morph}.  On the other hand}
 -\tilde{J}(f)-L^*(f) &=0\\
 \<D^2\alpha,\psi_1,\psi_2,\psi_3\> =\,& \rho(\psi_1)\<D\alpha,\psi_2,\psi_3\> -\<D\alpha,[\psi_1,\psi_2],\psi_3\> +\cp  \\
  =& [\rho(\psi_1),\rho(\psi_2)]\<\alpha,\psi_3\> -\rho(\psi_1)\<\alpha,[\psi_2,\psi_3]\> -\dots  \\
  &\dots-\rho[\psi_1,\psi_2]\<\alpha,\psi_3\> +\rho(\psi_1)\<\alpha,[\psi_2,\psi_3]\> +\dots \\
  &\dots+\<\alpha,[[\psi_1,\psi_2],\psi_3]\> +\cp  \\
  =& \<\alpha,-J(\psi_1,\psi_2,\psi_3)\> = \<-\tilde{J}(\alpha),\psi_1\wedge\psi_2\wedge\psi_3\>
\intertext{Note that on the other hand}
  L^*(\alpha) &= 0.  \\
  \<D^2v, \psi_1\wedge\psi_2\> &= \conn_{\psi_1}\<Dv,\psi_2\> -\conn_{\psi_2}\<Dv,\psi_1\> -\<Dv,[\psi_1,\psi_2]\>  \\
  =& [\conn_{\psi_1},\conn_{\psi_2}]v -\conn_{[\psi_1,\psi_2]}v  \\
  =& [\psi_1,[\psi_2,t(v)]] -[\psi_2,[\psi_1,t(v)]] -[[\psi_1,\psi_2],t(v)] \\
  =& J(t(v),\psi_1,\psi_2) = \<L(v),\psi_1\wedge\psi_2\>
\intertext{and therefore for $\beta\in\Gamma(F^*)$}
  \<D^2\beta,v\>_F &= D\<D\beta,v\>_F +\<D\beta, Dv\>_F \\
  =& D^2\<\beta,v\>_F -D\<\beta,D v\>_F +D\<\beta,Dv\>_F -\<\beta,D^2v\>_F  \\
\intertext{The first term vanishes due to the first step of the proof and the fourth term is exactly}
 =& -\<\beta, L(v)\>_F=\<-L^*(\beta),v\>_F.
\end{align*}
\end{proof}

The remaining part of the proof is
\begin{lemma}\label{l:Jd} $\hat\delta\circ\hat{J}+\hat{J}\circ\hat\delta = \tilde{J}+L^*$
\end{lemma}
\begin{proof}  Analogous to the previous lemmas we only need to check for $f\in\smooth(M)$, $\alpha\in\Gamma(E^*)$ and $\beta\in\Gamma(F^*)$.  Let further $\psi_i\in\Gamma(E)$ and $v\in\Gamma(F)$, then
\begin{align*}
  \hat{J}(f) &=0=\hat\delta(f)=\tilde{J}(f)=L^*(f).  \\
  \hat{J}(\alpha) &=0=L^*(\alpha)  \\
  \hat{J}(\hat\delta(\alpha)) &= \alpha(t(J(\cdot,\cdot,\cdot))= \tilde{J}(\alpha)
\intertext{This proves the statement for $\alpha\in\Gamma(E^*)$.}
  \hat\delta(\beta) &= 0 \\
  \hat{J}(\beta) &= \<\beta,J(\cdot,\cdot,\cdot)\>_F  \\
  \<\hat\delta(\hat{J}(\beta)),\psi_1\wedge\psi_2\otimes v\> &= \<\beta,J(\psi_1,\psi_2,t(v))\>_F
    = \<\beta,L(v,\psi_1,\psi_2)\>_F \\
    &= \<L^*(\beta),\psi_1\wedge\psi_2\otimes v\>
\end{align*}
\end{proof}
This completes the proof of the Theorem.

\subsection{The super picture}\label{ss:super}
For this subsection we assume that the reader is familiar with the concept of $\N$-manifolds.  For references see e.g.\ \cite[chap.\ 2]{Royt02}, \cite[sect.\ 2]{Sev01b} or \cite[sec.\ 4]{Vor01}.  For convenience we have added a small summary in the Appendix~\ref{a:Nmfd}.  We will moreover assume that $\rho$ is a regular map, i.e.\ $F:=\ker\rho\subset A$ a vector bundle (of constant rank).

Given the $\N$-manifold $\M:=A[1]\oplus F[2]$, then $\gsmooth(\M)\cong C^\bullet$ via $l''$, the above cochain complex.  With the Q-structure found in \cite{Gru10}, we have a differential on $\gsmooth(\M)$.  In coordinates $x^i$ on a chart of $M$, $\xi^a$ odd fiber coordinates on $A[1]$ over the same chart and $b^B$ fiber coordinates of $F[2]$ of degree 2, $Q$ reads as follows
\[\begin{split}  Q= \rho^i_a(x)\xi^a\pfrac{}{x^i} -\tfrac12C_{ab}^c(x)\xi^a\xi^b\pfrac{}{\xi^c} +t^a_B(x)b^B\pfrac{}{\xi^a} \\
  -\Gamma_{aB}^C(x)\xi^ab^B\pfrac{}{b^C}  -\tfrac16J_{abc}^B(x)\xi^a\xi^b\xi^c\pfrac{}{b^B}
\end{split}
\]  where for the local frame $\{e_a: a=1,\dots,\rk E\}$ of $E$ dual to the coordinates $\xi^a$, and coordinate vector fields $\{\partial/\partial x^i: i=1,\dots,\dim M\}$ the anchor map encodes as $\rho(e_a)=\rho^i_a(x)\partial/\partial x^i$.  The structure functions of the bracket are $[e_a,e_b]=C_{ab}^c(x)e_c$, the coefficients of the connection $\nabla$ are $\conn_{e_a}e_B=\Gamma_{aB}^C(x) e_C$ for a frame $\{e_B:B=1,\dots,\rk F\}$ of the second vector bundle $F\subset E$.  The vector bundle morphism $t\:E\to F$ is encoded via $t(e_a)=t_a^B(x) e_B$.  The Jacobiator finally is $J(e_a,e_b,e_c)=J_{abc}^B(x) e_B$.

As it was shown in \cite{Gru10}\nocite{GSr10}, $DJ=0$ together with the other axioms of almost Lie algebroids imply that $[Q,Q]=0$.  This $Q$-structure is what gives the differential in our cohomology, i.e.\
\[  Q[l''(\gamma)] = l''(\ud\gamma)
\] for $\gamma\in C^\bullet$ and $\ud=D+\hat\delta+\hat{J}$ as defined above.  Note that in this picture the maps $J$, $\tilde{J}$, $\hat{J}$, $L$ and $L^*$ are all represented by the same coefficients $(J_{abc}^B)$.  Analogously $t$ and $\hat\delta$ are both represented by $(t_a^B)$.

\section{Examples}\label{s:ex}
Given an almost Lie algebroid $A$ and the definition of the differential $\ud$ from Theorem~\ref{t:cohom}, we can also define a cohomology $H^\bullet(A,\ud)$ in the usual way.  In this section we want to investigate to what it computes.  Throughout these computations let either $\rho$ be a regular map (and thus $F:=\ker\rho\subset A$ a vector bundle) or at least the base field $\k$ be of characteristic 0.\footnote{This asserts us that all sequences of modules split.}

\begin{lemma}\label{l:deltaCohom}  Let $V$ be a vector space endowed with bracket $[\cdot,\cdot]=0$ and anchor map $\rho=0$, then $\ud=\hat\delta$ has cohomology $H^0(V,\hat\delta)=\R$ and $H^{>0}(V,\hat\delta)=0$.
\end{lemma}
This is an easy exercise in graded algebra.

Remember the notion of an almost Lie algebra as in Example~\subref{x:1}{x:aLiealg}.
 
\begin{prop}  Given an almost Lie algebra $\g$, then the cohomology from Section~\ref{s:cohom} is $H^0(\g,\ud)=\R$ and $H^{>0}(\g,\ud)=0$.
\end{prop}
\begin{proof}  The anchor map is $\rho=0$ and therefore $A=F=\g$ in the notation of Section~\ref{s:cohom}.  Now we decompose $\wedge^\bullet\g^*\otimes \Symm^\bullet\g^*=\bigoplus_{i\ge0} E_i$ by the degree of the symmetric part.  Correspondingly $e\in \wedge^\bullet\g^*\otimes \Symm^\bullet\g^*$ with $\ud e=0$ decomposes as $e=\sum_i e_i$ with $e_i\in E_i$.  Note that $D\:E_i\to E_i$, $\hat\delta\:E_i\to E_{i+1}$, and $\hat{J}\:E_i\to E_{i-1}$.  Together with \eqref{ud} this gives
\[ 0 = De_i +\hat\delta e_{i-1} +\hat{J} e_{i+1}  \label{eachDegree}
\] for all $i\ge0$.  Let $N$ denote the top degree of $e$, i.e.\ all $e_i=0$ for $i>N$.  Then \eqref{eachDegree} for $i=N+1$ implies
\begin{align*}
  0 &= \hat\delta e_N
\intertext{and as long as $N>0$, there is an $r_N\in E_{N-1}$ such that}
  e_N &= \hat\delta r_N.
\intertext{For $i=N$ \eqref{eachDegree} implies}
 0 &= De_N+\hat\delta e_{N-1}  \\
  &= D\hat\delta r_N +\hat\delta e_{N-1}= \hat\delta(-D r_N+e_{N-1})
\intertext{as long as $N-1>0$ there is an $r_{N-1}\in E_{N-2}$ with}
  e_{N-1} &= D r_N +\hat\delta r_{N-1}\\
\intertext{and for $i=N-1$}
  0 &= De_{N-1} +\hat\delta e_{N-2} +\hat{J} e_N \\
   &= D^2 r_N +D\hat\delta r_{N-1} +\hat\delta e_{N-2} +\hat{J}\hat\delta r_N
\intertext{Proposition~\ref{p:D^2} and Lemma~\ref{l:Jd} imply}
  &= \hat\delta(-D r_{N-1}+e_{N-2}-\hat{J}(e_N))
\intertext{therefore as long as $i=N-2>0$ there is an $r_{N-2}\in E_{N-3}$ with}
  e_i &= D r_{i+1} +\hat\delta r_i +\hat{J}e_{i+2}.
\intertext{This goes through for all $i>0$ until}
 0 &= D e_1 +\hat\delta e_0 +\hat{J} e_2  \\
  &= \hat\delta(-D r_1+e_0 -\hat{J} r_2)
\intertext{But with $r_0=0$ there is an $f\in \wedge^0\g^*=\R$ with}
 e_0 &= D r_1+\hat\delta r_0+\hat{J} r_2 +f.
\intertext{Note that $r_1\in E_0$ and therefore}
  \hat{J}(r_1) &= 0.
\intertext{adding all up, we obtain}
 e &= \sum_{i=0}^N e_i = \sum_{i=0}^N (D r_{i+1}+\hat\delta r_i +\hat{J} r_{i+2})+f   \\
 &= \sum_{i=1}^N D r_i +\sum_{i=0}^N \hat\delta r_i +\sum_{i=2}^N \hat{J}r_i +f
\end{align*}
{But actually we can extend the first and the last sum to $i=0$.  Therefore $e$ is exact up to a constant term $f\in\R$.}
\end{proof}

The following example shows that the cohomology is non-trivial.
\begin{ex}  Let $A=TM$ be the tangent bundle of a smooth manifold.  Thus the anchor map is $\rho=\id_{TM}$ with $F=\ker\rho=M\times0$ the 0-bundle and therefore $\hat\delta=0=\hat{J}$.  But then $\ud=D$ and $C^\bullet(A,\rho)=\Omega_M^\bullet(M)$ the forms and $D=\ud_M$ the de Rham differential.  Therefore in this case $H^\bullet(A,\ud)=H_{dR}^\bullet(M)$ the de Rham cohomology which is non-trivial, e.g.\ for the sphere.
\end{ex}

Remember that given an integrable distribution $L\subset TM$, then the leaf-wise cohomology $H_l^\bullet(L)$ is the Lie algebroid cohomology of $L$.  In particular $\ud_L$ the Lie algebroid differential is given by Equation \eqref{D}.  In degree $0$ it consists of functions constant along the integral leaves of $L$.

These examples suggest the following proposition.
\begin{prop}\label{p:cohom}  Given a regular almost Lie algebroid $(A,\rho)$ over $M$ with anchor map $\rho$ and $L:=\im\rho\subset TM$, then $H^\bullet(A,\ud)\cong H_l(L)$.
\end{prop}
Note that $L$ is indeed integrable, because of the morphism property \eqref{morph} for almost Lie algebroids.
\begin{proof}  With the notation from Section 3, we can write $A\cong L\oplus F$ and therefore $\wedge^\bullet A^*\cong \wedge^\bullet L^*\otimes \wedge^\bullet F^*$.  The analogon of Lemma~\ref{l:deltaCohom} is
\[  H^\bullet(A,\hat\delta) = \Gamma(\wedge^\bullet L^*).
\]  With the notation from the previous proof, $\Gamma(\wedge^\bullet L^*)\subset E_0$ and therefore the arguments of this proof apply, except for $e_0$ where it modifies to $r_0=0$ and there exists an $\alpha\in\Gamma(\wedge^\bullet L^*)$ with
\begin{align*}
  e_0 &= Dr_1 +\hat\delta r_0+\hat{J}r_2+\alpha.
\intertext{again $\hat{J}r_1=0$ and therefore Equation~\eqref{eachDegree} for $i=-1$ implies}
  0 &= D\alpha.
\intertext{But then we have added $e\in C^\bullet(A,\rho)$ to}
  e &=\sum_{i=0}^N e_i =\ud\left(\sum_{i=0}^N r_i\right)+\alpha.
\intertext{Note that for these $\alpha\in\Gamma(\wedge^\bullet L^*)$, $\hat\delta\alpha=0=\hat{J}\alpha$ and therefore}
 \ud\alpha &= D\alpha=\ud_L\alpha
\end{align*}  the leaf-wise differential.  So the claim follows.
\end{proof}
\begin{rem} Suppose we have an $H$-twisted Poisson structure $\Pi\in\Gamma(\wedge^2TM)$ as in Example~\subref{x:2}{x:tPoisson} with $L=\im\Pi^\#$.  So $L$ Gives a singular but generalized integrable distribution.  Then the construction of Section~\ref{s:cohom} defines a cohomology for the almost Lie algebroid $(T^*M,\Pi^\#,[\cdot,\cdot]_{\Pi,H})$ and Proposition~\ref{p:cohom} suggests that it is related to $L$ itself.
\end{rem}

\appendix
\section{Odd vector fields on $\N$-manifolds}\label{a:Nmfd}
\begin{defn} An $\N$-manifold is a ringed space $(M,\gsmooth)$ where $M$ is a smooth manifold and $\gsmooth$ a sheaf of $\N$-graded commutative associative algebras such that for sufficiently small open neighborhoods $U\subset M$:
\[ \gsmooth(U)\cong \smooth(U)\otimes \wedge^\bullet \R^{p_1}\otimes \Symm^\bullet \R^{p_2}\otimes\wedge^\bullet\dots \R^{p_d}
\] for some fixed non-negative integers $d\in\N$ and $p_i\in\N$, $i=1,\dots,d$.  $\wedge^\bullet\R^p$ for $p\in\N$ is the Grassmann algebra in $p$ variables and $\Symm^\bullet\R^q$ is the symmetric (i.e.\ polynomial) algebra in $q\in\N$ variables.  The tensor product is that of $\N$-graded algebras where the elements of $\smooth(U)$ have degree 0, the generators $\R^{p_1}$ of $\wedge^\bullet\R^{p_1}$ degree 1, those $\R^{p_2}$ of $\Symm^\bullet\R^{p_2}$ degree 2 and so on.
\end{defn}
We denote the $\N$-grading of $\gsmooth$ as $\gsmooth=\bigoplus_{n\ge0} \gsmooth_n$ and the degree of $f\in\gsmooth_n$ as $|f|=n$.  If above $p_d$ is not zero, we call $d$ the maximum degree of $\M=(M,\gsmooth)$.
\begin{ex}  Let $M$ be a smooth manifold and $A\to M$ and $F\to M$ two vector bundles. $\M=A[1]\oplus F[2]$ with structure sheaf
 \[ \gsmooth(U)=\Gamma_U(\wedge^\bullet A^*\otimes \Symm^\bullet F^*)\]
 where $\Gamma_U$ denotes the sections over the open subset $U\subset M$ is an $\N$-manifold in the above sense.  Here $\Gamma(A^*)$ have degree 1 and $\Gamma(F^*)$ degree 2.  It is then obvious that the structure sheaf $\gsmooth$ is generated by $\gsmooth_0=\smooth(M)$, $\gsmooth_1=\Gamma(A^*)$ and $\Gamma(F^*)\subset\gsmooth_2$.
\end{ex}

\begin{defn}  The vector fields $\X$ are the graded derivations of the structure sheaf $\gsmooth$, i.e.\ for $X\in\X$ there is a degree $|X|\in\Z$ such that $X\:\gsmooth_{p}\to\gsmooth_{p+|X|}$ and for all $f\in\gsmooth_{|f|}$, $g\in\gsmooth$
 \[ X[fg]= X[f]g+ (-1)^{|X|\,|f|}f X[g].
 \]
\end{defn}
If $\M$ has maximum degree $d$, then vector fields are graded as $\X=\bigoplus_{n=-d}^\infty \X_n$ with every $X\in \X_n$ has degree $n=|X|$.

\begin{prop}  If $X\in\X_{|X|}$ and $Y\in\X_{|Y|}$, then $[X,Y]$ which for $f\in\gsmooth$ is defined as
 \[ [X,Y][f]:= X[Y[f]] -(-1)^{|X|\,|Y|}Y[X[f]]
 \] is a vector field of degree $|X|+|Y|$.
\end{prop}
\begin{proof}  Straight-forward calculations.
\end{proof}
In particular Lemma~\ref{l:commut} is a special case of this for $X$ and $Y$ odd.\par\bigskip

A \strong{Q-structure} is now a vector field $Q\in\X_1$ such that $2Q^2=[Q,Q]=0$.  Note that this implies in particular that $Q$ is a differential of the algebra $\gsmooth$.

\begin{ex}\begin{enumerate}\item  Starting from $\M=T[1]M$ we see that $\gsmooth=\Omega_M^\bullet(M)$.  Therefore the de Rham differential $Q=\ud_M$ is an example of a Q-structure.   The cohomology of this Q-structure is therefore the de Rham cohomology.

\item  Another example arises with $\M=A[1]$ the odd vector bundle of a Lie algebroid.  Here $\gsmooth=\Gamma(\wedge^\bullet A^*)=\Omega_M^\bullet(A)$ and therefore the Lie algebroid differential $\ud_A$ defined as in \eqref{D} gives a Q-structure $Q=\ud_A$.  Its cohomology is therefore the Lie algebroid cohomology.
\end{enumerate}
\end{ex}

\newcommand{\etalchar}[1]{$^{#1}$}

\end{document}